\documentclass{amsart}
\usepackage[utf8]{inputenc}

\usepackage{amsmath,amsfonts,amsthm, amssymb}
\usepackage[hidelinks]{hyperref}
\usepackage[abbrev, msc-links, nobysame]{amsrefs}
\usepackage{amsaddr}

\numberwithin{equation}{section}

\theoremstyle{plain}
\newtheorem{lemma}{Lemma}[section]
\newtheorem{theorem}[lemma]{Theorem}
\newtheorem{prop}[lemma]{Proposition}

\newtheorem*{claim}{Claim}

\theoremstyle{definition}

\newtheorem{defn}[lemma]{Definition}

\theoremstyle{remark}
\newtheorem{remark}[lemma]{Remark}
\newtheorem*{remark*}{Remark}

\makeatletter
\def\@set@authors@addresses{\par
  \@setauthors%
  \begingroup
  \centering
  \def\address##1##2{\par\addvspace\bigskipamount%
    {\itshape\ignorespaces##2}%
  }%
  \def\email##1##2{%
    \@ifnotempty{##2}{, \ignorespaces{\ttfamily##2}}}%
  \def\curraddr##1##2{}%
  \def\urladdr##1##2{}%
  \addresses
  \endgroup
\par }
\renewcommand{\email}[2][]{%
  \ifx\emails\@empty\relax\else{\g@addto@macro\emails{,\space}}\fi%
  \@ifnotempty{#1}{\g@addto@macro\emails{\textrm{#1:}\space}}%
  \g@addto@macro\emails{#2}%
}
\makeatother

\title[Local ellipsoid characterization and isometric subspaces problem]{Local Blaschke--Kakutani ellipsoid characterization and Banach's isometric subspaces problem}

\author{Sergei Ivanov$^{a, b}$}
\email[S.\ Ivanov]{svivanov@pdmi.ras.ru}

\author{Daniil Mamaev$^{a, c}$}
\email[D.\ Mamaev (corresponding author)]{daniil.mamaev.21@ucl.ac.uk}

\author{Anya Nordskova$^{a, d}$}
\email[A.\ Nordskova]{anya.nordskova@uhasselt.be}

\address{$^a$St.\ Petersburg Department of Steklov Mathematical Institute,
Fontanka 27, St.\ Petersburg 191023, Russia}

\address{$^b$St.\ Petersburg State University,
Universitetskaya nab.\ 7-9,
St.\ Petersburg 199034, Russia}

\address{$^c$University College London, Gower Street, London WC1E 6BT, United Kingdom}

\address{$^d$Universiteit Hasselt, Agoralaan, gebouw D, 3590 Diepenbeek, Belgium}

\subjclass[2010]{46C15, 52A21}
\keywords{Ellipsoid characterization, convex body, cross-section}

\newcommand{\R}{\mathbb R}

\newcommand{\U}{\mathcal U}

\renewcommand{\P}{\mathbb P}

\newcommand{\Gr}{\mathrm{Gr}}

\newcommand{\id}{\mathrm{id}}
\newcommand{\res}[2]{\left. #1 \right|_{#2}}
\DeclareMathOperator{\tr}{Trace}

\DeclareMathOperator{\linspan}{LinSpan}

\DeclareMathOperator{\Int}{Int}

\DeclareMathOperator{\im}{im}
\DeclareMathOperator{\pr}{pr}

\DeclareMathOperator{\Hom}{Hom}
\newcommand{\la}{\lambda}
\newcommand{\pd}{\partial}
\newcommand{\co}{\colon}
\newcommand{\ep}{\varepsilon}

\begin{document}

\begin{abstract}
We prove the following local version of Blaschke--Kakutani's
characterization of ellipsoids:
Let $V$ be a finite-dimensional real vector space,
$B\subset V$ a convex body with 0 in its interior,
and ${2\le k<\dim V}$ an integer.
Suppose that the body
$B$ is contained in a cylinder based on the cross-section $B \cap X$
for every $k$-plane $X$ from a connected open set of linear $k$-planes in $V$.
Then in the region of $V$ swept by these $k$-planes 
$B$ coincides with either an ellipsoid, or a cylinder over an ellipsoid, or a cylinder over a $k$-dimensional base.

For $k=2$ and $k=3$ we obtain as a corollary
a local solution to Banach's isometric subspaces problem:
If all cross-sections of $B$ by $k$-planes
from a connected open set are linearly equivalent,
then the same conclusion as above holds.
\end{abstract}

\maketitle

\section{Introduction}

A classical result by Kakutani
(\cite{Kakutani39}, see also \cite{Gruber}*{Theorem 12.5})
characterizes Euclidean spaces among finite-dimensional
normed spaces as follows:

\medskip{\emph{
Let $V=(V,\|\cdot\|)$ be an $n$-dimensional normed space and $2\le k<n$.
Suppose that for every $k$-dimensional linear subspace $X$ of $V$,
there exists a linear projector $P_X\colon V\to X$ onto $X$
with unit operator norm, i.e.\ such that $\|P_X(v)\|\le \|v\|$
for all $v\in V$. Then $\|\cdot\|$ is a Euclidean
(i.e., an inner product) norm.
}}
\medskip

This fact is also known as the Blaschke--Kakutani characterization;
it can be seen as the dual form of Blaschke's characterization
of ellipsoids via planarity of shadow boundaries.
The projector property always holds for 1-dimensional subspaces
(by the Hahn--Banach theorem), which explains the condition $k\ge 2$.

In this paper we characterize norms for which the same assumption
is satisfied locally, that is for subspaces $X$ ranging over
an open subset of the respective Grassmannian.
The answer is that the local structure of the norm near these subspaces
is either Euclidean or cylindrical,
see Theorem~\ref{t: local Blaschke--Kakutani} below for the precise formulation.

Our main motivation and application is a local version
of a low-dimensional solution of Banach's problem
about normed spaces where all subspaces
of a fixed dimension are isometric.
Using the approach from \cite{IMN23}
together with Theorem~\ref{t: local Blaschke--Kakutani}
we show that the answer
to the local version of Banach's problem for $k=2,3$ is the same
as in Theorem~\ref{t: local Blaschke--Kakutani},
see Theorem~\ref{t: local Banach}.

In the case when the norm is smooth and strictly convex,
the local Blaschke--Kakutani characterization was
obtained by Calvert \cite{Calvert87}.
The cylindrical case does not appear in \cite{Calvert87} because of
the strict convexity assumption.
The local Banach's isometric subspaces problem for $k=2$
and smooth strictly convex norms was solved in~\cite{I18}
as a part of the proof of a Finsler geometry result.
Note that the Blaschke--Kakutani characterization easily reduces
to the case $k=2$ while Banach's problem does not.

\subsection*{Definitions and formulations}

We state our results in convex geometry terms
rather than in terms of norms.
As usual, a norm on a vector space $V$ is represented by its unit ball $B$,
which is a convex body in~$V$.
Since a part of our motivation comes from Finsler geometry,
we do not assume that $B$ is symmetric
and hence consider norms that are not necessarily symmetric
(``Minkowski norms'', see section~\ref{s:prelim}).
The existence of a norm non-increasing projector to
a linear subspace $X\subset V$ is equivalent to the
property that $B$ is contained in a cylinder with base
$B\cap X$, see Definition \ref{d: cylinder} and
Lemma \ref{l: characterisations of cylindricity} below.

By $\Gr_k V$ we denote the Grassmannian of $k$-dimensional
linear subspaces of a vector space $V$.
Two linear subspaces $X,Y\subset V$ are called \emph{complementary}
if $V=X+Y$ and $X\cap Y=0$.
A \textit{convex body} is a compact convex set
with a nonempty interior.

\begin{defn}\label{d: cylinder}
Let $V$ be a real $n$-dimensional vector space
and $1\le k<n$ an integer.
A set $C\subset V$ is called a \emph{$k$-cylinder}
if it can be represented in the form
$$
C=K+Y
$$
where $K$ is a $k$-dimensional convex body
in a linear subspace $X\in\Gr_kV$,
and $Y$ is a linear subspace complementary to~$X$.

The set $K$ is referred to as a \emph{base}
and $Y$ as a \emph{generatrix} of~$C$.
If the value of $k$ is clear from context,
we omit it and call $k$-cylinders simply \emph{cylinders}.
\end{defn}

Note that the generatrix of a cylinder is unique but a base is not.
In fact, if $C$ is a $k$-cylinder then
for every linear subspace $X'\in\Gr_kV$
such that the set $K'=C\cap X'$ is compact, $K'$ is a base of~$C$.
(The compactness of $K'$ is equivalent to $X'\cap Y=0$ where $Y$ is the generatrix).

Our first result is the following theorem.

\begin{theorem} \label{t: local Blaschke--Kakutani}
Let $V$ be a real $n$-dimensional vector space,
$B \subset V$ a convex body containing 0 in its interior,
$2 \le k < n$ an integer, and $\U \subset \Gr_k V$
a nonempty connected open set.

Suppose that for every $X\in\U$
the body $B$ is contained in a $k$-cylinder with base $B\cap X$.
Then there exists a set $B'\subset V$ such that
\begin{equation} \label{e: two bodies coincide in a neighbourhood}
    B\cap X=B'\cap X \quad\text{for all $X\in\U$}
\end{equation}
and at least one of the following holds:
\begin{enumerate}
\item $B'$ is a $k$-cylinder;
\item $B'=\{v \in V \colon Q(v) \le 1\}$ for some
nonnegative definite quadratic form $Q$ on $V$.
\end{enumerate}
\end{theorem}

The case (1) in Theorem \ref{t: local Blaschke--Kakutani}
occurs if all $k$-cylinders from the assumption
have the same generatrix.
In~(2), the set $B'$ may be an ellipsoid or,
if $Q$ is degenerate, a cylinder over
an $m$-dimensional ellipsoid for some $m\ge k$.
The cylindrical and degenerate cases are unavoidable
in the local setting.
In fact, any $B$ satisfying the conclusion of the theorem
satisfies its assumption,
see Lemma \ref{l: local kakutani is if and only if}.

\begin{remark}\label{rem: nondegeneracy}
If $\U$ in Theorem~\ref{t: local Blaschke--Kakutani}
is the entire Grassmannian $\Gr_kV$ then
\eqref{e: two bodies coincide in a neighbourhood}
implies that $B'=B$.
Then, since $B$ is compact, the cylindrical and degenerate cases
are ruled out and the only remaining option is a sublevel set
of a positive definite quadratic form.
Thus Theorem \ref{t: local Blaschke--Kakutani}
implies the original Blaschke--Kakutani characterization
and moreover generalizes it to non-symmetric norms.

We also note that the cylindrical and degenerate cases cannot occur
if $B$ is strictly convex or, more generally,
if the union of the subspaces from $\U$ contains
an extreme point of~$B$.
(Recall that an \emph{extreme point} of $B$
is a point $p\in B$ such that $B\setminus\{p\}$ is convex).
\end{remark}

The proof of Theorem \ref{t: local Blaschke--Kakutani} is given in
sections~\ref{s: proof in 3D}~and~\ref{s: proof induction}
after technical preparations
in sections~\hbox{\ref{s:prelim}--\ref{s: local cylinders}}.
In section~\ref{s: proof in 3D} we handle the case $n=3$
and in section~\ref{s: proof induction}
the proof is finished by induction on dimension.

The ellipsoid and cylinder cases in dimension~3
are separated by whether the correspondence between crossing planes and cylinders containing the body is 1-to-1 or not.
In the latter case we deduce the result via convex geometry arguments,
and in the former case the proof is based on a local version
of the fundamental theorem of projective geometry.
This is similar to proofs in \cite{Calvert87} and \cite{I18}.

\subsection*{Banach's isometric subspaces problem}

In 1932, Banach \cite{Banach32} posed the following problem:

\medskip{\emph{ 
Let $V=(V,\|\cdot\|)$ be a normed vector space
and $2\le k<\dim V$ an integer.
Suppose that all $k$-dimensional linear subspaces of $V$ are isometric.
Is $\|\cdot\|$ necessarily an inner product norm?
}}\medskip

The problem translates into the language of convex geometry as follows:
Consider a convex body $B\subset V$ (the unit ball of a norm) and suppose
that all cross-sections of $B$ by $k$-dimensional linear subspaces
are \textit{linearly equivalent}, that is, for every $X,Y\in\Gr_k V$
there exists a linear map $L\colon X\to Y$ such that $L(B\cap X)=B\cap Y$.
The question is whether such a body is necessarily a centered ellipsoid.

It is answered affirmatively in some dimensions
and remains open in others.
For a long time the only known result was the solution for $k=2$
by Auerbach, Mazur, and Ulam~\cite{AMU35}.
Then Dvoretzky \cite{Dvo59} solved the problem
for infinite-dimensional spaces
and Gromov \cite{Gro67} settled the case of even $k$ and
the case $\dim V\ge k+2$ for all $k$.
Recently the problem was solved for
$k\equiv 1 \bmod 4$ except $k=133$ by 
Bor, Hern\'{a}ndez Lamoneda, Jim\'{e}nez-Desantiago, and Montejano~\cite{BHJM21}
and for $k=3$ by the authors \cite{IMN23}.

The proofs in \cites{AMU35,Gro67,BHJM21} rely on algebraic topology
of Grassmannians to find obstructions to the existence
of certain families of linear equivalences.
In contrast, the proof for $k=3$ in \cite{IMN23} is based on differential geometric
analysis in a neighborhood of a single cross-section.
This suggests that it makes sense to consider
a local version of the problem where the linear equivalence is assumed
only for a small open set of cross-sections.
Such a local result
was obtained in \cite{I18} for $k=2$, $n=3$ and a smooth strictly convex body;
the conclusion is that the respective part of the body coincides with an ellipsoid.

In this paper we extend the results of \cite{I18} and \cite{IMN23}
and solve the local version of Banach's problem for $k=2$ and $k=3$.
Like in the case of Theorem \ref{t: local Blaschke--Kakutani},
the problem admits locally cylindrical solutions
in addition to locally ellipsoidal ones.

\begin{theorem}\label{t: local Banach}
Let $V$ be an $n$-dimensional real vector space, $B\subset V$
a convex body containing 0 in its interior,
$k\in\{2,3\}$,
and $\U \subset \Gr_k V$ a nonempty connected open set.
Suppose that for every $X_1,X_2\subset\U$
the cross-sections $B\cap X_1$ and $B\cap X_2$
are linearly equivalent.

Then the same conclusion as in Theorem \ref{t: local Blaschke--Kakutani}
holds, namely there exists $B'\subset V$ such that
$B\cap X=B'\cap X$ for all $X\in\U$
and $B$ is either a $k$-cylinder or a sublevel set
of a nonnegative definite quadratic form (or both).
\end{theorem}

The proof of Theorem \ref{t: local Banach} is a combination
of Theorem \ref{t: local Blaschke--Kakutani} and the results of \cite{IMN23}.
In fact, the key propositions in \cite{IMN23} show that
the assumptions of Theorem \ref{t: local Banach} imply
those of Theorem~\ref{t: local Blaschke--Kakutani}.
See section \ref{sec: proof_local_Banach} for details.

\begin{remark} \label{r: local reconstruction from sections}
    A related question is whether a convex body in $\R^n$ is uniquely determined, up to a symmetry or homothety, by the congruence or affine types of its $k$-dimensional cross-sections through the origin.
    For congruence types this question goes back to Nakajima~\cite{N32} and S\"{u}ss~\cite{S32}.
    In general the answer is negative as shown by Zhang~\cite{Z18}.
    However in the symmetric case an application of the spherical Radon transform shows that the answer is affirmative, moreover a symmetric body is uniquely determined by the areas of its cross-sections.
    
    This type of questions can be localized as well, for example, one may ask the following:
    
    \emph{Let $B_1$ and $B_2$ be origin symmetric bodies in $\R^n$, ${2 \le k < n}$, and $\U \subset \Gr_k\R^n$ an open set.
    Suppose that for every $X \in \U$ the sections $B_1 \cap X$ and $B_2 \cap X$ are congruent (or, more generally, are linearly equivalent and have the same area).
    Is it true that $B_1 \cap X = B_2 \cap X$ for all $X \in \U$?}

    Note that Proposition~\ref{p: ellipses only assemble into ellipsoid} implies an affirmative answer to this question when one of the bodies is an ellipsoid. 
    Also see Purnaras--Saroglou~\cite{PS21} for a local problem of a similar flavor.
\end{remark}
\section{Preliminaries}
\label{s:prelim}

\subsection{Notation and conventions}
In this paper, a ``vector space''
always means a finite-dimensional real vector space
and a ``subspace'' means a linear subspace. 
For vector spaces $X$ and $Y$, 
$\Hom(X,Y)$ 
denotes the space of linear maps from $X$ to~$Y$, and 
$X^*=\Hom(X,\R)$ 
the dual space to~$X$.
For a vector space $V$, the Grassmann manifold 
$\Gr_k V$ 
consists of (unoriented) $k$-dimensional subspaces of $V$. 
For 
$X \in \Gr_k V$ 
and  $m\ge k$ we denote by $\Gr_m(V,X)$ the set of
subspaces from $\Gr_m V$ containing~$X$:
\begin{equation} \label{e: Gr_m(V, W)}
    \Gr_m(V, X) := \{W \in \Gr_m V \colon X \subset W\}.
\end{equation}

For a subset $S \subset V$ we denote by $\linspan S$ the smallest linear subspace of $V$ containing $S$. For $v \in V \setminus \{0\}$ the line through $v$ is denoted by $\R v = \{\lambda v \mid \lambda \in \R\}$.

For complementary subspaces $X, Y \subset V$
we denote by $\pr^Y_X$ the projection to $X$ along $Y$:
\begin{equation} \label{e: pr^L_H}
    \pr^Y_X \colon V \to V, \quad v \mapsto \text{ the unique point in } (v + Y) \cap X
\end{equation} 
Note that any linear projector (i.e., idempotent linear map)
$P \in \Hom(V, V)$ is uniquely of the above form, with $X = \im P$ and $Y = \ker P$.

For a convex set $K\subset V$ we denote by $\pd K$
the relative boundary of $K$, that is the boundary in the topology of its affine span. A convex set $B \subset V$ is called a \emph{convex body} if it is compact and has nonempty interior. By an {\em ellipsoid} we mean the unit ball of an inner product norm in a vector space.
In other words, all ellipsoids are assumed to be centered at~$0$.
The same terminology adjustment applies to ellipses in dimension~2.

A {\em Minkowski seminorm} on a vector space $V$ is a
function $\Phi\co V\to\R_+$ which is positively 1-homogeneous
and subadditive (and hence convex).
A {\em Minkowski norm} is a Minkowski seminorm
which is positive on $V\setminus\{0\}$.
The difference from usual norms is that a Minkowski norm
is not assumed symmetric. 
There is a standard bijection between Minkowski
seminorms and convex sets with 0 in the interior.
Namely, to each Minkowski seminorm $\Phi$ one associates its unit ball
$B_\Phi = \{ x\in V : \Phi(x)\le 1\}$, and for every convex
set $B \subset V$ with $0$ in the interior there
is a corresponding Minkowski seminorm
$\Phi^B(x) = \inf\{\lambda > 0 \mid x/\lambda \in B\}$.
Note that $\Phi^B$ is a Minkowski norm if and only if $B$ is compact.

Recall that a \emph{supporting hyperplane} of a convex body $B$
at a point $p\in\pd B$ is an affine hyperplane $H\ni p$ such that
$H\cap\Int B=\varnothing$.
We say that a point $p\in\pd B$ is a \emph{smooth point} of $\pd B$
if $B$ has a unique supporting hyperplane at~$p$.
In this case the hyperplane is also called
the tangent hyperplane of $B$ at~$b$.
If $B$ is a unit ball of a Minkowski norm $\Phi$ than the smoothness
of a point $p\in\pd B$ is equivalent to the property that $\Phi$
is differentiable at~$p$.
Note that a Minkowski norm is differentiable almost everywhere
(since it is a convex function) and hence almost all points
of the boundary of a convex body are smooth points.

\subsection{Assumptions and assertions of Theorem \ref{t: local Blaschke--Kakutani}}

\begin{defn} \label{d: contracting subspace}
Let $\Phi$ be a Minkowski seminorm on a vector space $V$.
A linear subspace $X \subset V$ is called
\emph{$\Phi$-contracting} if there exists a linear projector $P$
from $V$ onto $X$ such that $\Phi(P(v))\le \Phi(v)$ for all $v\in V$.

Recall that every such linear projector $P$
is of the form $\pr_X^Y$ for some subspace $Y$
complementary to~$X$.
We refer to $Y$ as a \emph{contracting direction} for~$X$.
\end{defn}

Being $\Phi$-contracting is a closed condition:
For every $k\le n$ the set of $k$-dimensional $\Phi$-contracting
subspaces is closed in $\Gr_k V$.
Also note that if $X$ is $\Phi$-contracting then $X$
is $(\res{\Phi}{W})$-contracting for every subspace $W\subset V$ containing $X$
(to prove this, just restrict $P$ to~$W$).

The following lemma provides various reformulations for the assumption
of Theorem~\ref{t: local Blaschke--Kakutani}.
It will be handy throughout the proof.

\begin{lemma} \label{l: characterisations of cylindricity}
Let $\Phi$ be a Minkowski seminorm on a vector space $V$, $B$ its unit ball,
and $X, Y \subset V$ complementary subspaces.
Then the following conditions are equivalent. 
\begin{enumerate}
    \item $B \subset (B \cap X) + Y$;
    \item $\pr^Y_X B = B \cap X$;
    \item $X$ is $\Phi$-contracting with contracting direction $Y$;
    \item $(p + Y) \cap \Int B = \varnothing$ for all $p \in \pd B \cap X$.
\end{enumerate}
\end{lemma}

\begin{proof}

$(1) \Rightarrow (2).$
The inclusion $B \cap X \subset \pr^Y_X B$ is trivial.
The reverse one follows from~(1)
and the identity $(\pr^Y_X)^{-1}(B\cap X) = (B\cap X)+Y$.

$(2) \Rightarrow (3).$
We have to show that $\Phi(\pr_X^Y(v))\le\Phi(v)$ for every $v\in V$.
If $\Phi(v)=1$ then $v\in B$, therefore $\pr_X^Y(v)\in B$ by (2),
hence $\Phi(\pr_X^Y(v))\le 1$.
If $\Phi(v)>0$ then the desired inequality follows by homogeneity.
If $\Phi(v)=0$ then $tv\in B$ for all $t\ge 0$,
this and (2) imply that $t\pr_X^Y(v)\in B$ for all $t\ge 0$,
therefore $\Phi(\pr_X^Y(v))=0$.
    
$(3) \Rightarrow (4).$ If $q \in (p + Y) \cap \Int B$ for some $p \in \pd B \cap X$, then $\pr^Y_X q = p$ hence
$$
    1 > \Phi(q) \overset{(3)}{\ge} \Phi\left(\pr^Y_X q\right) = \Phi(p) = 1
$$
and we obtain a contradiction.
    
$(4) \Rightarrow (1).$
Suppose that (1) is false and pick $b_0\in B$ such that $b_0\notin (B \cap X) + Y$.
Let $p_0 = \pr^Y_X(b_0)$, then $p_0\notin B$, hence $\Phi(p_0)>1$.
Let $p = p_0/\Phi(p_0)$, then $p\in \pd B \cap X$.
Now observe that the set $p + Y$ contains a point
$q = b_0/\Phi(p_0)$ which belongs to $\Int B$,
contrary to~(4).
\end{proof}

The assumption of Theorem \ref{t: local Blaschke--Kakutani} says
that for every $X\in\U$ the condition (1)
from Lemma~\ref{l: characterisations of cylindricity}
is satisfied for some $Y=Y_X\in\Gr_{n-k}V$.
In view of Lemma \ref{l: characterisations of cylindricity}(3),
this can be restated as follows: every $X\in\U$ is $\Phi$-contracting,
where $\Phi$ is the Minkowski norm associated to~$B$.

\begin{defn} \label{d: coincide near and locally cylindrical}
Let $V$ be a vector space, $B_1, B_2 \subset V$ two convex sets with zero
in the interiors, and $X \in \Gr_k V$.
We say that $B_1$ and $B_2$ \emph{coincide near} $X$ if
there exists an open set $U\subset V$ such that
$X\subset U$ and $B_1 \cap U = B_2 \cap U$.

A convex body $B\subset V$ is called \emph{locally cylindrical near $X$}
if there exists a $k$-cylinder $C$ such that $B$ and $C$
coincide near~$X$.
Note that in this case $B\cap X$ is a base of $C$ (see Definition~\ref{d: cylinder})
since $C\cap X=B\cap X$ and $B\cap X$ is compact.
\end{defn}

The next lemma shows that the assumption of
Theorem~\ref{t: local Blaschke--Kakutani} is local.

\begin{lemma}\label{l: locality of Kakutani condition}
Let $V$ be an $n$-dimensional vector space, $X\in\Gr_k V$,
and $B_1,B_2\subset V$ two convex sets with zero in the interiors.
Assume that $B_1$ and $B_2$ coincide near $X$
and $B_1$ is contained in a $k$-cylinder $C$ with base $B_1\cap X$.
Then $B_2\subset C$ as well.
\end{lemma}

\begin{proof}
Let $K=B_1\cap X=B_2\cap X$ and $C=K+Y$ where $Y$ is
a subspace complementary to~$X$.
Suppose to the contrary that $B_2\not\subset C$.
Then by Lemma~\ref{l: characterisations of cylindricity}
there exist $p\in\pd K$ and $q\in B_2$ such that
$$
q\in (p + Y)\cap\Int B_2 .
$$
Since $p\in B_2$ and $q\in\Int B_2$,
for every $\ep\in(0,1)$ the point $p_\ep:=p+\ep(q-p)$
belongs to $\Int B_2$.
Since $B_1$ and $B_2$ coincide near $X$,
it follows that $p_\ep\in\Int B_1$ for a sufficiently small $\ep$.
On the other hand, $p_\ep\in p+Y\subset\pd C$.
This contradicts the assumption that $B_1\subset C$.
\end{proof}

In the last lemma of this section we show that Theorem~\ref{t: local Blaschke--Kakutani} is in fact an if-and-only-if statement.

\begin{lemma}\label{l: local kakutani is if and only if}
Let $V$ be an $n$-dimensional vector space,
$B\subset V$ a convex body with zero in the interior, and $X\in\Gr_kV$.
Assume that at least one of the following holds:
\begin{enumerate}
 \item
$B$ is locally cylindrical near $X$;
\item
$B$ coincides with $B'=\{v\in V: Q(v)\le 1\}$ near $X$,
where $Q$ is a nonnegative definite quadratic form on~$V$.
\end{enumerate}
Then $B$ is contained in a $k$-cylinder with base $B\cap X$.
\end{lemma}

\begin{proof}
First assume (1) and let $C$ be the corresponding cylinder
(see Definition \ref{d: coincide near and locally cylindrical}).
The desired property follows from
Lemma \ref{l: locality of Kakutani condition}
applied to $C$ and~$B$ in place of $B_1$ and~$B_2$, respectively.

Now assume (2) and let $Y$ be the orthogonal complement to $X$
with respect to the symmetric bilinear form associated to~$Q$.
Since $B'\cap X=B\cap X$ is compact, $\res{Q}{X}$ is positive definite
and therefore $Y$ is a complementary subspace to~$X$.
From the orthogonality we have $Q(\pr^Y_X v) \le Q(v)$ for all $v \in V$,
therefore $B'$ is contained in a cylinder $C=(B'\cap X)+Y$.
Applying Lemma \ref{l: locality of Kakutani condition}
to $B'$ and $B$ finishes the proof.
\end{proof}

\section{Quadratic forms}

The goal of this section is to prove the following local version of the well-known fact
that a normed space is Euclidean if all of its subspaces of fixed dimension $k \ge 2$ are Euclidean. 
Though the statement looks standard, we could not find it in the literature
and the proof is not so immediate as one might expect.

\begin{prop} \label{p: ellipses only assemble into ellipsoid}
    Let $\Phi$ be a Minkowski norm on an $n$-dimensional vector space $V$, 
    $2 \le k < n$ 
    an integer, and 
    $\U \subset \Gr_k V$ 
    a connected nonempty open set.
    Suppose that for every $X\in\U$ the restriction
    $\res{\Phi}{X}$  is an inner product norm on $X$.

    Then there exists a unique quadratic form  $Q$ on $X$ such that
    $\left(\res{\Phi}{X}\right) ^ 2 = \res{Q}{X}$ 
    for all 
    $X \in \U$. 
    Moreover $Q$ is nonnegative definite.
\end{prop}

The following notation and terminology will be handy throughout the proof.
For a basis 
${\mathbf v = (v_1,\dots,v_n)}$ 
of a vector space $V$ we denote by 
$\Pi^{\mathbf v}_{ij}$ 
its coordinate planes:
$$
 \Pi^{\mathbf v}_{ij} = \linspan\{v_i,v_j\}, \qquad 1\le i\ne j\le n .
$$
If 
$\U\subset\Gr_2V$ 
is an open set of planes, 
we say that a basis $\mathbf v$ is \textit{$\U$-compatible} if all its coordinate planes 
$\Pi^{\mathbf v}_{ij}$
belong to~$\U$.
Clearly for any plane 
$\Pi\in\U$ 
every basis $(v_1,v_2)$ of $\Pi$ can be extended
to a $\U$-compatible basis of~$V$ 
(just choose the remaining vectors sufficiently close to $\Pi$).

We precede the proof of
Proposition \ref{p: ellipses only assemble into ellipsoid}
with a couple of lemmas.

\begin{lemma}\label{l: frame of ellipses}
Let 
$\mathbf v=(v_1,\dots,v_n)$ 
be a basis of a vector space $V$ and 
$F\colon V\to\R$ 
a function whose restrictions to the coordinate planes 
$\Pi^{\mathbf v}_{ij}$
are quadratic forms on these planes.
Then there exists a unique quadratic form $Q$ on $V$ such that
$\res{Q}{\Pi^{\mathbf v}_{ij}}=\res{F}{\Pi_{ij}}$ for all $i\ne j$.
\end{lemma}

\begin{proof}
Let $(x_1, \ldots, x_n)$ be the coordinates on $V$
with respect to the basis $\mathbf v$.
We construct the (symmetric) matrix $(c_{ij})$ of $Q$
in these coordinates from the values of $F$ on the coordinate planes.

First define $c_{ii} = F(v_i)$ for all $1\le i\le n$.
Then for each pair $i,j$ with $i\ne j$,
consider the quadratic form $Q_{ij}:=\res{F}{\Pi^{\mathbf v}_{ij}}$
on the plane $\Pi^{\mathbf v}_{ij}$. Since $Q_{ij}(v_i)=F(v_i)=c_{ii}$
and $Q_{ij}(v_j)=F(v_j)=c_{jj}$,
the coordinate expression of $Q_{ij}$ has the form
$$
Q_{ij}(x_i,x_j)=c_{ii}x_i^2 + c_{jj} x_j^2 + 2 c_{ij} x_i x_j
$$
for some $c_{ij}\in\R$.
We use this expression to define $c_{ij}$.

The resulting quadratic form $ Q(x_1, \ldots, x_n) = \sum_{i,j} c_{ij}x_ix_j$
satisfies $\res{Q}{\Pi^{\mathbf v}_{ij}} = Q_{ij}$ for all $i\ne j$.
The uniqueness is obvious from the construction.
\end{proof}

The next lemma essentially covers the case $n=3$
of Proposition \ref{p: ellipses only assemble into ellipsoid}.
Note that in this case we do not assume that $\U$ is connected.

\begin{lemma}\label{l: ellipses in 3D}
Let $X$ be a $3$-dimensional vector space,
$F\colon X\to\R$ 
a continuous function, and 
$\U\subset\Gr_2V$ 
a nonempty open set.
Suppose that for every 
$\Pi\in\U$ 
the restriction
$\res{F}{\Pi}$ 
is a quadratic form on $\Pi$.
Then there exists a unique quadratic form $Q$ on $X$ such that 
$\res{F}{\Pi}=\res{Q}{\Pi}$ for all $\Pi\in\U$.
\end{lemma}

\begin{proof}
Fix a $\U$-compatible basis $\mathbf v=(v_1,v_2,v_3)$ of $X$.
By Lemma~\ref{l: frame of ellipses}, there exists a unique quadratic
form $Q$ on $X$ that coincides with $F$ on the coordinate planes $\Pi^{\mathbf v}_{ij}$.
We show that this $Q$ satisfies
$\res{F}{\Pi}=\res{Q}{\Pi}$ for all $\Pi\in\U$.

First consider a plane $\Pi\in\U$ which is \textit{generic}
in the sense that it does not contain any of the vectors $v_1,v_2,v_3$.
Since $F$ and $Q$ coincide on the coordinate planes,
they coincide on the lines
$$
 \ell_{ij} := \Pi \cap \Pi^{\mathbf v}_{ij}, \qquad 1\le i<j\le 3 .
$$
Both $\res{F}{\Pi}$ and $\res{Q}{\Pi}$ are quadratic forms on $\Pi$,
and a quadratic form on $\Pi$ is uniquely determined by its values
on the three distinct lines $\ell_{ij}$.
Hence $\res{F}{\Pi}=\res{Q}{\Pi}$ if $\Pi$ is generic.
To finish the proof, observe that any non-generic plane $\Pi$
can be approximated by generic ones
and the identity $\res{F}{\Pi}=\res{Q}{\Pi}$ follows by continuity.
\end{proof}

\begin{proof}[Proof of Proposition \ref{p: ellipses only assemble into ellipsoid}]

Let $V$, $\Phi$ and $\U\subset\Gr_k V$ be as in
Proposition \ref{p: ellipses only assemble into ellipsoid}.
Define $F=\Phi^2$ and
$$
 \Omega = \bigcup_{X\in\U} X\setminus\{0\}.
$$
It is easy to see that $\Omega$ is a connected open subset of~$V$.

The assertion of the proposition can be rewritten as follows:
there exists a unique quadratic form $Q$ on $V$
such that $\res{F}{\Omega}=\res{Q}{\Omega}$
and furthermore $Q$ is nonnegative definite.
First we verify the uniqueness and nonnegative definiteness of such $Q$.
The uniqueness follows from the facts
that $\Omega$ is open and a quadratic form is uniquely
determined by its restriction to any open set.

Now suppose that $Q$ is a quadratic form such that
$\res{F}{\Omega}=\res{Q}{\Omega}$ and $Q(v)<0$ for some $v\in V$.
Fix $p\in\Omega$ and define $f(t)=F(p+tv)$ for all $t\in\R$.
The function $f$ is convex since $F=\Phi^2$
is a convex function on~$V$.
The identity  $\res{F}{\Omega}=\res{Q}{\Omega}$ implies that
$f(t)=Q(p+tv)$ for all $t$ sufficiently close to~0.
Therefore $f$ is smooth near~0 and $f''(0) = 2Q(v)<0$,
contrary to the convexity of $f$.
This contradiction shows that $Q$ must be nonnegative definite.

It remains to prove the existence of a quadratic form $Q$ such that
$\res{F}{\Omega}=\res{Q}{\Omega}$.
First we reduce this statement to the special case when $k=2$.
Consider the set
$$
\U_2 = \{\Pi \in \Gr_2 V \colon \Pi \subset X \text{ for some } X \in \U\}
$$
and observe that $\U_2$ is a connected open subset of $\Gr_2 V$,
$\res{\Phi}{\Pi}$ is a quadratic form for every $\Pi \in \U_2$,
and $\bigcup_{\Pi\in\U_2}\Pi\setminus\{0\}=\Omega$.
Thus it suffices to prove the proposition for $k=2$ and $\U_2$ in place of $\U$.
We therefore assume $k=2$ for the rest of the proof.

For a $\U$-compatible basis $\mathbf v = (v_1,\dots,v_n)$ of $V$,
we denote by $Q^{\mathbf v}$ the quadratic form on $V$ satisfying
$\res{F}{\Pi^{\mathbf v}_{ij}} = \res{Q^{\mathbf v}}{\Pi^{\mathbf v}_{ij}}$
for all $1 \le i < j \le n$.
Such a form exists and is unique by Lemma~\ref{l: frame of ellipses}.
Clearly $Q^{\mathbf v}$ does not change if the vectors of $\mathbf v$
are permuted or multiplied by nonzero scalars.

\begin{claim}
Let $v_1,\dots,v_n\in V$ and $t\in\R$ be such that
$\mathbf v = (v_1, v_2, \ldots, v_n)$ and
$$
\mathbf v' = (v_1+tv_2, v_2, \ldots, v_n)
$$
are $\U$-compatible bases.
Then $Q^{\mathbf v} = Q^{\mathbf{v'}}$.
\end{claim}

\begin{proof}
By the definition of $Q^{\mathbf{v'}}$ it suffices to show that
\begin{equation} \label{e: old Qv on new coordinate planes}
\res{Q^{\mathbf v}}{\Pi^{\mathbf v'}_{ij}} = \res{F}{\Pi^{\mathbf v'}_{ij}}
\end{equation}
for all $1\le i<j\le n$.
Observe that $\Pi^{\mathbf v'}_{ij}=\Pi^{\mathbf v}_{ij}$
if $i,j\ge 2$ or $\{i,j\}=\{1,2\}$,
so \eqref{e: old Qv on new coordinate planes} trivially holds in these cases.
It remains to verify \eqref{e: old Qv on new coordinate planes}
for $i=1$ and $j>2$.
Fix $j>2$ and apply Lemma \ref{l: ellipses in 3D}
to the 3-dimensional subspace
$$
 X = \linspan\{v_1,v_2,v_j\},
$$
the set $\U\cap\Gr_2 X$ in place of $\U$,
and the function $\res{F}{X}$ in place of~$F$.
This yields a quadratic form $Q$ on $X$ such that
$\res{Q}{\Pi}=\res{F}{\Pi}$ for all planes $\Pi\in \U\cap\Gr_2 X$.
In particular $Q$ and $F$ coincide
on the planes
$\Pi^{\mathbf v}_{12}$, $\Pi^{\mathbf v}_{1j}$ and~$\Pi^{\mathbf v}_{2j}$,
therefore $Q=\res{Q^{\mathbf v}}{X}$
by the uniqueness part of Lemma \ref{l: frame of ellipses}.
On the other hand, $Q$ and $F$ coincide on the plane
$\Pi^{\mathbf v'}_{1j}=\linspan\{v_1+tv_2,v_j\}$
since this plane also belongs to $\U\cap \Gr_2 X$.
Therefore \eqref{e: old Qv on new coordinate planes} holds for all 
$1\le i<j\le n$
and Claim follows.
\end{proof}

Fix $p\in\Omega$ and choose a $\U$-compatible
basis $\mathbf v = (v_1,\dots,v_n)$ such that $v_1=p$.
Fix $\ep>0$ so small that for every point $p'\in V$ of the form
\begin{equation} \label{e: variatiion of v1}
 p' = p + \sum_{i=1}^n t_i v_i
 \quad \text{where $t_i\in(-\ep,\ep)$ for all $1\le i\le n$},
\end{equation}
the collection $\mathbf v'=(p',v_2,\dots,v_n)$
is a $\U$-compatible basis.
We are going to show that $Q^{\mathbf v'}=Q^{\mathbf v}$
for every such $\mathbf v'$.

Let $p'\in V$ be as in \eqref{e: variatiion of v1}.
Connect $p$ to $p'$ by a sequence $p_0=p,p_1,\dots,p_n=p'$
where
$$
 p_m =  p + \sum_{i=1}^m t_i v_i, \qquad m=1,\dots, n ,
$$
and let $\mathbf v^m=(p_m,v_2,\dots,v_n)$ for each~$m$.
By the choice of $\ep$, each $\mathbf v^m$ is a $\U$-compatible basis.
Observe that $Q^{\mathbf v^1} = Q^{\mathbf v}$ since $\mathbf v^1$
is obtained from $\mathbf v$ by rescaling the first basis vector.
For $2\le m \le n$, the basis $\mathbf v^{m}$ is obtained
from $\mathbf v^{m-1}$ by a transformation as in Claim
(up to a permutation of indices),
hence $Q^{\mathbf v^{m}} = Q^{\mathbf v^{m-1}}$.
Thus
$$
Q^{\mathbf v}=Q^{\mathbf v^1}=Q^{\mathbf v^2}
=\dots=Q^{\mathbf v^n}=Q^{\mathbf v'} .
$$
In particular, since $p'$ is one of the basis vectors in $\mathbf v'$,
we have $Q^{\mathbf v}(p')=Q^{\mathbf v'}(p')=F(p')$.

Thus $Q^{\mathbf v}(p')=F(p')$ for any point $p'$ of the form \eqref{e: variatiion of v1}.
The range of such points $p'$ is an open neighborhood of~$p$,
therefore we have proven the following statement
(where $Q^{\mathbf v}$ is renamed to~$Q_p$):
For every $p\in\Omega$ there exists a quadratic form $Q_p$ on $V$
such that $Q_p$ and $F$ coincide in a neighborhood of~$p$.
Since a quadratic form is uniquely determined by its restriction
to any open set, such $Q_p$ is unique for every $p\in\Omega$
and the map $p\mapsto Q_p$ is locally constant on~$\U$.
Since $\Omega$ is connected, it follows that
all $Q_p$, $p\in\Omega$, are one and the same quadratic form.
Denote this quadratic form by $Q$ and observe that $Q(p)=Q_p(p)=F(p)$
for all $p\in\Omega$.
Thus $\res{Q}{\Omega}=\res{F}{\Omega}$
and Proposition \ref{p: ellipses only assemble into ellipsoid} follows.
\end{proof}

\section{Local cylinders}
\label{s: local cylinders}

In this section we collect technical facts about locally cylindrical
convex bodies,
see Definition~\ref{d: coincide near and locally cylindrical}.
Throughout this section $V$ is an $n$-dimensional vector space,
$\Phi$ a Minkowski norm on~$V$,
and $B$ is the unit ball of~$\Phi$.

\begin{lemma} \label{l: two cylinders with the same generatrix coincide}
Let $X_1,X_2\in \Gr_k V$ and $Y\in\Gr_{n-k} V$ be such that
$$
 B \subset (B \cap X_i) + Y \text{ for } i = 1, 2 .
$$
Then $(B \cap X_1) + Y = (B \cap X_2) + Y$.
\end{lemma}

\begin{proof}
The assumption of the lemma implies that
$$
 (B \cap X_1) + Y \subset B + Y
 \subset \left((B \cap X_2) + Y\right) + Y = (B \cap X_2) + Y .
$$
Swapping $X_1$ and $X_2$ yields the opposite inclusion,
hence the result.
\end{proof}

\begin{lemma}\label{l: same cylinder}
Let 
$\U\subset \Gr_kV$ 
be an open set and $X_0\in\U$.
Suppose that all subspaces from $\U$ are $\Phi$-contracting with the same contracting direction 
$Y_0\in\Gr_{n-k}V$
(see Definition \ref{d: contracting subspace}).
Then $B$ is locally cylindrical near~$X_0$.
\end{lemma}

\begin{proof}
By Lemma~\ref{l: characterisations of cylindricity} we have
$ B\subset (B\cap X)+Y_0$ for all $X\in\U$.
Then 
Lemma~\ref{l: two cylinders with the same generatrix coincide} 
implies that all cylinders 
$(B\cap X)+Y_0$, $X\in\U$,
are in fact one and the same cylinder, which we denote by~$C$.
Let $U\subset V$ be the union of $\Int B$ and all subspaces from $\U$.
The set $U$ is open, contains $X_0$,
and satisfies $B\cap U=C\cap U$ since
$
 C\cap X = ((B\cap X)+Y_0)\cap X = B\cap X
$
for every $X\in\U$.
Thus $B$ and~$C$ coincide near $X_0$
hence $B$ is locally cylindrical near~$X_0$.
\end{proof}

\begin{lemma}\label{l: propagation of local cylindricity}
Let $\U\subset\Gr_k V$ be a nonempty connected open set.
Suppose that $B$ is locally cylindrical near~$X$
for every $X \in \U$.
Then there exists a $k$-cylinder $C$ such that
$B\cap X=C\cap X$ for all $X \in \U$.
\end{lemma}

\begin{proof}
First we show that for every $X\in\U$,
a $k$-cylinder that coincides with $B$ near~$X$
is unique.
Indeed, suppose that for some $X\in\U$
there are two such cylinders, $C_1$ and $C_2$,
and observe that $C_1$ and $C_2$ coincide near~$X$.
By Lemma~\ref{l: locality of Kakutani condition} applied
to $C_2$ in place of $B_2$ and $C_1$ in place of both $B_1$ and $C$,
it follows that $C_2\subset C_1$.
Similarly $C_1\subset C_2$, hence $C_1=C_2$,
showing the desired uniqueness.

Let $C_X$ denote the above unique cylinder.
Pick $X\in\U$ and let $U\subset V$ be a neighborhood of $X$
such that $B\cap U=C_X\cap U$
(see Definition~\ref{d: coincide near and locally cylindrical}).
Since $B$ is compact, there exists a neighborhood $\U_X\subset\U$ of~$X$
such that $B\cap X'\subset U$ for all $X'\in\U_X$.
Then $C_X$ and~$B$ coincide near $X'$ for every $X'\in\U_X$.
Hence $C_{X'}=C_X$, by the above uniqueness applied to $X'$ in place of~$X$.

Thus the map $X\mapsto C_X$ is locally constant and hence constant on $\U$.
Denote the constant cylinder $C_X$ by~$C$
and observe that $B\cap X=C_X\cap X=C\cap X$
for every $X\in\U$.
\end{proof}

The next lemma provides a convenient reformulation
of the conclusion of Theorem \ref{t: local Blaschke--Kakutani}.

\begin{lemma}\label{l: relaxed theorem conclusion}
Let $\U\subset\Gr_k V$ be a nonempty connected open set.
Suppose that for every $X\in\U$ at least one of the following holds:
\begin{enumerate}
  \item $B$ is locally cylindrical near $X$;
  \item $B \cap X$ is an ellipsoid.
\end{enumerate}
Then the conclusion of Theorem \ref{t: local Blaschke--Kakutani}
holds for $B$ and~$\U$.
\end{lemma}

\begin{proof}
If all intersections $B \cap X$ are ellipsoids then
$\res{\Phi}{X}$ is an inner product norm for every $X\in\U$
(recall that all ellipsoids in this paper are 0-centered).
Then Proposition \ref{p: ellipses only assemble into ellipsoid}
implies that there exists a quadratic form $Q$ on $V$
such that $(\res{\Phi}{X})^2 = \res{Q}{X}$
or, equivalently, $B\cap X= B'\cap X$
for all $X\in\U$
where $B'=\{v\in V: Q(v)\le 1\}$.
This is the second option in the conclusion of
Theorem~\ref{t: local Blaschke--Kakutani}.

Now assume that some of the intersections $B\cap X$, $X\in\U$,
are not ellipsoids.
Let $\U_0$ be a connected component of the nonempty open set
$$
\U':=\{X\in U: \text{$B\cap X$ is not an ellipsoid} \} .
$$
By our assumptions $B$ is locally cylindrical near $X$
for every $X\in\U_0$.
By Lemma \ref{l: propagation of local cylindricity}
there exists a $k$-cylinder $C_0$ such that
$B\cap X = C_0\cap X$
for all $X\in\U_0$.
If $\U_0=\U$ then this is the first option in the conclusion
of Theorem~\ref{t: local Blaschke--Kakutani}.
It remains to rule out the case when $\U_0\ne\U$.

Suppose that $\U_0\ne\U$.
Then, since $\U$ is connected, there exists $X_1\in\U\setminus\U_0$
that belongs to the closure of $\U_0$.
Clearly $X_1\notin\U'$, hence $B\cap X_1$ is an ellipsoid.
On the other hand, all intersections of the form $B\cap X$ where $X\in\U_0$,
are linearly equivalent since they are compact cross-sections
of the same cylinder~$C_0$.
Since the linear equivalence is a closed condition,
it follows that $B\cap X_1$ is linearly equivalent to $B\cap X_0$ for any $X_0 \in \U_0$.
However $B\cap X_0$ is not an ellipsoid, a contradiction.
\end{proof}

The next lemma will allow us to reduce the theorems to
the codimension~1 case.

\begin{lemma}\label{l: locally cylindrical in subspaces}
Let $X\in\Gr_kV$ be such that for every $W\in\Gr_{k+1}(V,X)$
the intersection $B\cap W$ is locally cylindrical near~$X$.
Then $B$ is locally cylindrical near~$X$.
\end{lemma}

\begin{proof}
Let $K=B\cap X$.
Choose $W_1,\dots,W_{n-k}\in\Gr_{k+1}(V,X)$ such that
$\linspan\left(\bigcup_{i=1}^{n-k} W_i\right) = V$.
For each $i=1,\dots,n-k$,
since $B\cap W_i$ is locally cylindrical near $X$,
there exists $v_i\in W_i\setminus X$
such that 
$K + [-v_i,v_i] = B \cap (X + [-v_i, v_i])$
where $[-v_i,v_i]$ denotes the straight line segment
between $-v_i$ and~$v_i$.
Define 
$Y=\linspan\{v_i\}_{i = 1}^{n - k}$ 
and $C=K+Y$,
then $Y$ is a subspace of~$V$ complementary to~$X$
and $C$ is a $k$-cylinder with base~$K$. 

We show that $B \subset C$. 
Lemma~\ref{l: characterisations of cylindricity} 
implies that it is enough to check that 
$(p + Y) \cap \Int B = \varnothing$ 
for all $p \in \pd K$. 
Fix $p \in \pd K$ and let $H$ be a supporting hyperplane to $B$ at $p$. 
Then for any 
$i = 1, \ldots, n -k$ 
the line 
$l_i = H \cap \linspan\{p, v_i\}$ 
is a supporting line to 
$B_i = B \cap \linspan\{p, v_i\}$ 
at $p$. 
As $p + [-v_i, v_i] \subset \pd B_i$, 
we have $p + v_i \in l_i$. 
Therefore $p + v_i \in H$
for all 
$i = 1, \ldots, n -k$ 
hence 
$p + Y \subset H$ 
thus 
$(p + Y) \cap \Int B = \varnothing$.

Define a neighborhood $U_0$ of 0 in $Y$ by
$$
 U_0 = \left\{ \sum_{i=1}^{n-k} t_i v_i  :
 t_1,\dots,t_{n-k}\in\R, \ \sum|t_i|< 1 \right\}
$$
and let $U=X+U_0$.
The choice of $v_i$ and the convexity of $B$ imply that 
$K+U_0\subset B$.
Combining this with $B \subset C$ we obtain
$B\cap U \subset C \cap U = K + U_0 \subset B \cap U$
hence $B$ is locally cylindrical near~$X$.
\end{proof}

\section{Proof of Theorem \ref{t: local Blaschke--Kakutani} in dimension~3}
\label{s: proof in 3D}

In this section we prove Theorem \ref{t: local Blaschke--Kakutani}
for $n=3$ and $k=2$.
Let $V$, $B$, and $\U$ satisfy the assumptions
of Theorem~\ref{t: local Blaschke--Kakutani} with $n=3$ and $k=2$.
That is, $V$ is a 3-dimensional vector space,
$B\subset V$ is a convex body with zero in the interior,
$\U\subset\Gr_2V$ is a nonempty connected open set,
and for every $X\in\U$ there exists a line $L_X\subset\Gr_1V$
such that $B$ is contained in the 2-cylinder with base $B\cap X$
and generatrix $L_X$:
\begin{equation}\label{e: 3D cylinder assumption}
 B \subset (B\cap X) + L_X .
\end{equation}
Note that \eqref{e: 3D cylinder assumption} implies that
$L_X$ is complementary to $X$, otherwise the set on
the right-hand side would be two-dimensional and could not contain~$B$.

We fix the above notation and assumptions for the rest of this section
and denote by~$\Phi$ the Minkowski norm associated to~$B$.

\medskip

To facilitate understanding, we first sketch the proof
in the case of the classical Kakutani criterion (that is, $\U=\Gr_2V$)
for a smooth, strictly convex, 0-symmetric body~$B$.
In this case one can define a continuous bijection $\psi\colon\Gr_1V\to\Gr_2V$
as follows: for $L\in\Gr_1V$,
$\psi(L)$ is the plane from $\Gr_2V$  parallel to the tangent
planes of $\pd B$ at the points of $L\cap\pd B$.
The Grassmannians $\Gr_1V$ and $\Gr_2V$ can be regarded as
real projective planes, which are dual to each other in the sense
that points in one projective plane correspond to lines in the other.
Indeed, a plane $X\in\Gr_2V$ corresponds to the set
$ \{ L\in\Gr_1 V : L \subset X \} $,
which is a line of the projective structure of~$\Gr_1V$,
and a line $L\in\Gr_1V$ corresponds to the set
$ \{ X\in\Gr_2 V : L \subset X \} $,
which is a line of the projective structure of~$\Gr_2V$.

The assumption of the Kakutani criterion implies that the above map $\psi$
sends projective lines of $\Gr_1V$ to projective lines of $\Gr_2V$.
Indeed, for every $X\in\Gr_2V$ and $L\in\Gr_1X$ the plane $\psi(L)$
must contain the line $L_X$ from~\eqref{e: 3D cylinder assumption},
thus $\psi$ sends the projective line of $\Gr_1X$ corresponding to $X$
to the projective line of $\Gr_2X$ corresponding to $L_X$.

By the fundamental theorem of projective geometry it follows that
$\psi$ is a projective map, that is, $\psi$ is induced by a linear
bijection $F \colon V \to V^*$ between the 3-dimensional vector spaces
with projectivizations $\P(V) = \Gr_1V$ and $\P(V^*) = \Gr_2V$.
In the latter case the projectivization
is given by the map 
$V^*\setminus\{0\}\to\Gr_2V, \; f\mapsto\ker f$.
To summarize, we obtain a linear bijection $F\colon V\to V^*$ such that
$\psi(\R v)=\ker(F(v))$ 
for all 
$v \in V \setminus \{0\}$.
It is not hard to show that such a linear parametrization of tangent
directions of $\pd B$ is possible only if $B$ is an ellipsoid.
(This last step is detailed in Lemma~\ref{l: 3D ellipsoid case} below).

\medskip

We now turn to the general case of Theorem~\ref{t: local Blaschke--Kakutani} for $n=3$ and $k=2$. 
The difference of the proof from the sketch above is that we have to tackle the cylindrical case
(see Lemma~\ref{l: 3D degenerate case} below)
and in the non-cylindrical case it is more natural to construct the projective dual of $\psi$.
The proof is composed of several lemmas.

\begin{lemma}\label{l: unique line}
For every $X\in\U$ there is a unique line $L_X\in\Gr_1(V)$
satisfying \eqref{e: 3D cylinder assumption}.
\end{lemma}

\begin{proof}
Suppose to the contrary that for some $X \in \U$ there exist two distinct lines $L_1 \neq L_2$ such that $B \subset K + L_i$, $i=1,2$,
where $K=B\cap X$.
Note that both $L_1$ and $L_2$ are complementary to $X$.

Pick a smooth point $p$ of $\pd K$ such that the
supporting line $l$ of $K$ at $p$ is not contained
in the plane $p+L_1+L_2$.
Then $l + L_1$ and $l+ L_2$ are two distinct supporting
planes of $B$ at~$p$.
Hence $B\subset H_1\cap H_2$ where $H_1$ and $H_2$
are closed half-spaces of $V$
bounded by $l+L_1$ and $l+L_2$ respectively.

Pick a plane $X'\in\U$ such that $p\in X'$ and $X'\ne X$.
Let $L=L_{X'}\in\Gr_1V$ be a line satisfying \eqref{e: 3D cylinder assumption}
for $X'$.
By Lemma~\ref{l: characterisations of cylindricity} we have
\begin{equation}\label{e: unique line 1}
B \cap X' = \pr_{X'}^L(B) \supset \pr_{X'}^L(K) .
\end{equation}
Assume for a moment that $L\not\subset X$.
Then the restriction of $\pr_{X'}^L$ to $X$
is a linear isomorphism between $X$ and~$X'$.
Since $\pr_{X'}^L(p)=p$, it follows that $p$
is a smooth point of $\pr_{X'}^L(K)$.
This and \eqref{e: unique line 1} imply that
$p$ is a smooth point of $B \cap X'$.
On the other hand, $B\cap X'$ is contained in the set
$X'\cap H_1\cap H_2$, which is a non-straight solid angle
with vertex at~$p$.
Hence $p$ is not a smooth point of $B\cap X'$, a contradiction.

This contradiction shows that the assumption $L\not\subset X$ is false.
Thus for every plane $X'\in\U$ such that $p\in X'$ and $X'\ne X$,
one has $L_{X'}\subset X$ (for any choice of $L_{X'}$).
Pick a sequence $\{X_i\}$ of such planes converging to~$X$.
For every~$i$ we have a line $L_{X_i}\subset X$
satisfying \eqref{e: 3D cylinder assumption} for $X_i$.
Passing to a subsequence if necessary we may assume
that the lines $L_{X_i}$ converge to some line $L_0\subset X$,
then $B\subset (B\cap X)+L_0$ by continuity.
However, the set $(B\cap X)+L_0$ is two-dimensional,
a contradiction.
This finishes the proof of Lemma~\ref{l: unique line}.
\end{proof}

With the help of Lemma~\ref{l: unique line},
we can now define a map $\varphi \colon \U \to \Gr_1(V)$
by $\varphi(X)=L_X$ where $L_X$ satisfies \eqref{e: 3D cylinder assumption}.
Since \eqref{e: 3D cylinder assumption} is a closed condition
on a pair $(X,L)\in\Gr_2V\times\Gr_1V$,
the uniqueness of~$L_X$ implies that $\varphi$ is continuous.
We fix the notation $\varphi$ for the rest of this section.

In the next lemma we handle the degenerate case when $\varphi$ is not injective.

\begin{lemma}\label{l: 3D degenerate case}
Suppose that the above map $\varphi$ is not injective.
Then there exists a 2-cylinder $C\subset V$
such that $B\cap X = C\cap X$ for all $X\in \U$.
\end{lemma}

\begin{proof}
Fix a line $L_0\in\Gr_1V$ having more than one $\varphi$-preimage.
For a point $p \in V$ define 
$$
\U_p = \{X \in \U \colon p \in \U\}
$$

\begin{claim}
Let $X_0,X_1\in\U$ be such that $X_0 \ne X_1$
and $\varphi(X_0) = \varphi(X_1) = L_0$.
Then there exists a point $p\in (\pd B\cap X_0)\setminus X_1$
such that $\varphi(X) = L_0$ for all $X \in \U_p$.
\end{claim}

\begin{proof}
Recall that every compact convex set in a finite-dimensional
vector space is a convex hull of its extreme points.
This implies that there exists an extreme point of $B\cap X_0$
outside the line $X_0 \cap X_1$. Let $p$ be such a point.

Let $q = \pr^{L_0}_{X_1} p$.
Then $q\ne p$ as $p\notin X_1$,
and $q\in B$ by Lemma~\ref{l: characterisations of cylindricity}.
Consider the set
$$
 \U_{pq} = \{ X\in\U :
 \text{$X\cap (p,q)$ is a single point} \}
$$
where $(p,q)$ denotes the open line segment between $p$ and~$q$.
We claim that $\varphi(X)=L_0$ for all $X\in\U_{pq}$.
Suppose to the contrary that 
$\varphi(X) = L \ne L_0$ 
for some $X\in\U_{pq}$.
Consider the points $p'=\pr_X^L(p)$
and $q'=\pr_X^L(q)$.
Since $p,q\in B$,
Lemma~\ref{l: characterisations of cylindricity}
implies that $p',q'\in B\cap X$.
The assumption that $L\ne L_0$ implies that $p'\ne q'$.
Thus $(p',q')$ is a nontrivial open line segment in $B\cap X$.
Moreover $(p',q')$ contains the intersection point of $X$ and $(p,q)$
since this point is preserved by~$\pr_X^L$.
Now consider points $p''=\pr_{X_0}^{L_0}(p')$ and $q''=\pr_{X_0}^{L_0}(q')$.
They belong to $B\cap X_0$
by Lemma~\ref{l: characterisations of cylindricity},
they are distinct since $p'\ne q'$ and $L_0$ is complementary to~$X$,
and we have $p\in(p'',q'')$ since $(p',q')$ contains a point
from $(p,q)\subset L_0$.
This contradicts the choice of $p$ as an extreme point of $B\cap X_0$.
This contradiction shows that 
$\varphi(X)=L_0$ for all $X\in\U_{pq}$.
Then Claim follows by continuity as every plane from $\U_p$ can be approximated by planes from $\U_{pq}$.
\end{proof}

Fix $X_0, X_1 \in \U$ such that $X_0 \ne X_1$ and
$\varphi(X_0) = \varphi(X_1) = L_0$.
Let $p\in (\pd B\cap X_0)\setminus X_1$
be a point provided by Claim.
Applying Claim to $X_0$ and any plane from $\U_p\setminus\{X_0\}$
we obtain another point $p'\in(\pd B\cap X_0)\setminus\R p$
such that $\varphi(X) = L_0$ for all $X \in \U_{p'}$.

Let $C=(B\cap X_0)+L_0$.
For every $X\in \U_p\cup\U_{p'}$ we have $\varphi(X)=L_0$
and hence $B\subset (B\cap X)+L_0$ by the definition of~$\varphi$.
This and
Lemma~\ref{l: two cylinders with the same generatrix coincide}
imply that $(B\cap X)+L_0=C$ and hence $B\cap X=C\cap X$
for all $X\in \U_p\cup\U_{p'}$.
Thus $B\cap U = C\cap U$
where $U\subset V$ is the union of all planes from $\U_p\cup\U_{p'}$.

Since $\R p\ne \R p'$,
$U$ contains an open neighborhood of $X_0\setminus 0$
and every plane $X\in\U$ sufficiently close to $X_0$
is contained in~$U$.
For every such plane $X\subset U$ we have
$$
 B\cap X = B\cap U \cap X = C\cap U \cap X =C\cap X,
$$
therefore $(B\cap X)+L_0 = C$
and hence $\varphi(X)=L_0$.
Thus $X_0$ has a neighborhood in $\U$ where $\varphi$ is constant.
Since $X_0$ is an arbitrary element of $\varphi^{-1}(L_0)$,
it follows that $\varphi^{-1}(L_0)$ is an open set.
Since $\U$ is connected, this implies that $\varphi$ is constant,
thus $\varphi(X)=L_0$ for all $X\in\U$.

Now  Lemma~\ref{l: two cylinders with the same generatrix coincide}
applied to $X_0$ and any $X\in\U$
implies that $(B\cap X)+L_0 = C$ and hence $B\cap X= C\cap X$,
finishing the proof of Lemma~\ref{l: 3D degenerate case}.
\end{proof}

Lemma \ref{l: 3D degenerate case} implies Theorem~\ref{t: local Blaschke--Kakutani}
in the case when $\varphi$ is not injective.
Now we consider the case when $\varphi$ is injective.

\begin{lemma}\label{l: collinearity}
Let $X_1,X_2,X_3\in\U$ be distinct planes
containing a common line $\ell\in\Gr_1V$.
Then the lines $\varphi(X_1),\varphi(X_2),\varphi(X_3)$
are contained in one plane from $\Gr_2V$.
\end{lemma}

\begin{proof}
First consider the case when $\ell\cap\pd B$ contains
a smooth point $p$ of~$\pd B$.
Let $T$ be the unique supporting plane of $B$ at~$p$.
By Lemma \ref{l: characterisations of cylindricity},
for every $j\in\{1,2,3\}$
the straight line $p+\varphi(X_j)$ does not intersect $\Int B$
and hence, by the smoothness of $B$ at~$p$,
this line is contained in~$T$.
Therefore the lines $\varphi(X_1),\varphi(X_2),\varphi(X_3)$
are contained in the plane from $\Gr_2V$ parallel to~$T$.
This proves the lemma in the case when $\ell\cap\pd B$ contains
a smooth point of~$\pd B$.

The general case follows by continuity, since smooth points
are dense in $\pd B$ and any triple of planes  $X_1,X_2,X_3\in\Gr_2V$
with $\ell=X_1\cap X_2\cap X_3\in\Gr_1V$
can be approximates by similar configurations
where intersection lines contain smooth points of~$\pd B$.
\end{proof}

Each of the Grassmannians $\Gr_1V$ and $\Gr_2V$ carries
a natural structure of a real projective plane
as explained in the beginning of this section.

Lemma \ref{l: collinearity} says that $\varphi$ preserves
collinearity with respect to these projective structures:
it sends any three collinear points of $\U\subset\Gr_2V$
to three collinear points of $\Gr_1V$.
We use the following generalization of
the fundamental theorem of projective geometry.

\begin{prop}[\cite{LTWZ}*{Theorem 3.2}]
\label{p: collineation is projective}
Let $U\subset\mathbb{RP}^2$ be a connected open set
and ${\varphi\co U\to\mathbb{RP}^2}$ an injective map
such that for any three collinear points of $x,y,z\in U$
their images $\varphi(x),\varphi(y),\varphi(z)$ are also collinear.
Assume that the image $\varphi(U)$ contains three non-collinear points.
Then $\varphi$ is the restriction of a projective map.
\end{prop}

If the map $\varphi\colon\U\to\Gr_1V$ defined above
is injective, then it satisfies the assumptions of
Proposition~\ref{p: collineation is projective}.
Indeed, $\varphi$ preserves collinearity by
Lemma~\ref{l: collinearity},
and its continuity and injectivity imply that
the image $\varphi(U)$ is not contained
in one projective line of $\Gr_1V$, hence $\varphi(U)$
contains three non-collinear points.
Thus there exists a projective map 
$\widetilde\varphi\colon\Gr_2V \to\Gr_1V $
such that 
$\res{\widetilde\varphi}{\U}=\varphi$. 

We now construct a dual projective map $\psi\colon \Gr_1 V \to \Gr_2 V$.
Pick $L\in\Gr_1V$ and consider the set $P_L=\{X\in\Gr_2V: L\subset X\}$.
It is a projective line in $\Gr_2V$,
hence its image $\widetilde\varphi(P_L)$ is a projective line in~$\Gr_1V$.
This means that
$\widetilde\varphi(P_L)=\Gr_1(P'_L)$ for some plane $P'_L\in\Gr_2V$.
We define $\psi(L)=P'_L$.
This yields an injective map $\psi\colon \Gr_1 V \to \Gr_2 V$
uniquely characterized by the property
\begin{equation}\label{e: dual projective map}
 L\subset X  \iff
 \widetilde\varphi(X) \subset \psi(L)
 \qquad \text{for all $L\in\Gr_1V$ and $X\in\Gr_2V$} .
\end{equation}
In particular \eqref{e: dual projective map}
implies that $\psi$ preserves collinearity,
therefore it is a projective map.
Hence $\psi$ is the projectivization of
some linear bijection $F\colon V\to V^*$
(recall that $\Gr_1V = \P(V)$ and $\Gr_2V = \P(V^*)$).
This means that
\begin{equation}\label{e: ker F(p)}
 \ker(F(p)) = \psi(\R p) \quad \text{ for all } p\in V\setminus\{0\}
\end{equation}

\begin{lemma}\label{l: F gives support planes}
Assume that $\varphi$ is injective and let $F\colon V\to V^*$ be as above.
Let $p\in\pd B$ be a point contained in at least one plane from~$\U$.
Then $p+\ker(F(p))$ is a supporting plane of $B$ at~$p$.
\end{lemma}

\begin{proof}
First assume that $p$ is a smooth point of $\pd B$. Let $T$
be the unique supporting plane of $B$ at~$p$
and $T_0\in\Gr_2V$ the plane parallel to~$T$ through the origin.
Pick distinct $X_1,X_2\in\U$ such that $p\in X_1\cap X_2$.
Similarly to the proof of Lemma~\ref{l: collinearity},
we have $p+\varphi(X_i)\subset T$ and hence $\varphi(X_i)\subset T_0$
for $i=1,2$.
On the other hand, \eqref{e: dual projective map} implies that
$\varphi(X_i)\subset\psi(X_1\cap X_2)$ for $i=1,2$.
Since there is only one plane containing $\varphi(X_1)$ and $\varphi(X_2)$,
it follows that $\psi(X_1\cap X_2)=T_0$.
This and \eqref{e: ker F(p)} imply that $\ker(F(p))=T_0$.
Thus $p+\ker(F(p))=p+T_0=T$.

We have shown that the assertion of the lemma holds in the case
when $p$ is a smooth point of~$\pd B$.
The general case follows by continuity.
\end{proof}

The next lemma, together with
Proposition~\ref{p: ellipses only assemble into ellipsoid},
proves Theorem~\ref{t: local Blaschke--Kakutani} for $n=3, k=2$
in the case when $\varphi$ is injective.

\begin{lemma}\label{l: 3D ellipsoid case}
Suppose that $\varphi$ is injective.
Then $B\cap X$ is an ellipse for every $X\in\U$ (recall that all ellipses in this paper are $0$-centered).
\end{lemma}

\begin{proof}
Fix $X\in\U$ and let $K=B\cap X$.
For each $p\in X$, define $f(p)\in X^*$ by $f(p)=\res{F(p)}{X}$,
where $F$ is the map from Lemma~\ref{l: F gives support planes}.
Then $f\co X\to X^*$ is a linear map.
Applying Lemma~\ref{l: F gives support planes} to $p\in\pd K$
we obtain that $f(p)\ne 0$ and
$p+\ker(f(p))$ is a supporting line of $K$ at~$p$.

Among other things this implies that there is a way to
continuously assign a supporting line to each point of~$\pd K$.
This is possible only if $\pd K$ is a $C^1$ curve
and the Minkowski norm $\Phi$ is $C^1$ away from~0.
Now for every $p\in\pd K$, the line $p+\ker(f(p))$
is the tangent line of $\pd K$ at~$p$.

We turn this family of supporting lines into a linear vector
field $W$ as follows.
Fix a nonzero skew-symmetric bilinear form $\omega$ on $X$,
and for each $p\in X$ let $W(p)\in X$ be the unique vector satisfying
$$
\omega(W(p), q) = f(p)(q) \text{ for all } q \in X.
$$
Then $W\colon X\to X$ is a non-degenerate linear map,
and for every $p\in\pd K$ we have $W(p) \in \ker(f(p))$,
hence the direction of $W(p)$ is the tangent direction of $\pd K$ at~$p$.
We have therefore obtained a linear vector field $W$ on $X$ that is tangent to $\pd K$.
It is well-known that the existence of such a vector field
implies that $K$ is an ellipse, see e.g. \cite{IMN23}*{Lemma 3.4}.
\end{proof}

We now compose the proof of Theorem~\ref{t: local Blaschke--Kakutani}
for $n=3$ and $k=2$ from the results of this section.
For $V$, $B$, $\U$ as in the theorem,
define a continuous map $\varphi\co\U\to\Gr_1V$ as explained after
Lemma~\ref{l: unique line}.
Then there are two cases: either $\varphi$ is injective or not.
If $\varphi$ is not injective then Lemma~\ref{l: 3D degenerate case}
implies that the alternative (1) of the conclusion
of Theorem~\ref{t: local Blaschke--Kakutani} takes place.
If $\varphi$ is injective then Lemma~\ref{l: 3D ellipsoid case}
implies that $B\cap X$ is an ellipse
and hence $\res{\Phi}{X}$ is a Euclidean norm for every $X\in\U$.
This and Proposition \ref{p: ellipses only assemble into ellipsoid}
imply that the alternative (2) of the conclusion
of Theorem~\ref{t: local Blaschke--Kakutani} takes place.
Thus Theorem~\ref{t: local Blaschke--Kakutani} holds for $n=3$ and $k=2$.

\section{Proof of Theorem \ref{t: local Blaschke--Kakutani} in higher dimensions}
\label{s: proof induction}

In this section we finish the proof of Theorem~\ref{t: local Blaschke--Kakutani}.
First we observe that for every fixed $n$ and~$k$
the statement of Theorem \ref{t: local Blaschke--Kakutani}
is equivalent to the following proposition.

\begin{prop} \label{p: local Kakutani}
Let $k\ge 2$ and $n\ge k+1$ be integers.
Let $V$ be an $n$-dimensional vector space,
$\Phi$ a Minkowski norm on $V$,
and $B$ the unit ball of $\Phi$.
Let $X_0 \in \Gr_k V$ be such that all
$k$-dimensional subspaces $X$ from a neighbourhood of $X_0$ in $\Gr_kV$
are $\Phi$-contracting (see Definition~\ref{d: contracting subspace}).
Then at least one of the following holds.
\begin{enumerate}
  \item $B$ is locally cylindrical near $X_0$
  (see Definition \ref{d: coincide near and locally cylindrical}).
  \item $B \cap X_0$ is an ellipsoid
  (recall that all ellipsoids in this paper are 0-centered).
\end{enumerate}
\end{prop}

To show that Theorem~\ref{t: local Blaschke--Kakutani}
is equivalent to Proposition~\ref{p: local Kakutani},
first observe that the assumptions on~$X$ in Theorem~\ref{t: local Blaschke--Kakutani}
and Proposition~\ref{p: local Kakutani} are equivalent
by Lemma~\ref{l: characterisations of cylindricity}.
The conclusion of Theorem~\ref{t: local Blaschke--Kakutani}
trivially implies that of Proposition~\ref{p: local Kakutani}.
Conversely, by Lemma~\ref{l: relaxed theorem conclusion}
the conclusion of Proposition~\ref{p: local Kakutani}
implies that of Theorem~\ref{t: local Blaschke--Kakutani}.

The proof of Proposition \ref{p: local Kakutani} occupies
the rest of this section.
We argue by induction with base $n=3$ and $k=2$ established
in section~\ref{s: proof in 3D}.
The induction step is based on the following lemma.

\begin{lemma} \label{l: dimension reduction}
Let $X_1,X_2\in\Gr_{n-1} V$ be two hyperplanes
and $L_1,L_2 \in \Gr_1 V$ two lines,
$X_1\ne X_2$ and $L_1\ne L_2$.
Suppose that $X_1$ and $X_2$ are $\Phi$-contracting
with contracting directions $L_1$ and $L_2$, resp.
(see Definition~\ref{d: contracting subspace}).

Then the subspace $X_1 \cap X_2 \in \Gr_{n-2} V$ is $\Phi$-contracting.
\end{lemma}

\begin{proof}
Define $W=X_1 \cap X_2$ and $Z=L_1+L_2$.
We are going to show that $W\cap Z=0$ and the projector $\pr^Z_W\colon V\to W$
does not increase~$\Phi$.

Consider the map 
$T=\pr^{L_1}_{X_1}\circ\pr^{L_2}_{X_2}$. 
Note that 
$\Phi(T(v)) \le \Phi(v)$ 
for all 
$v \in V$ 
as 
$\pr^{L_1}_{X_1}$ 
and 
$\pr^{L_2}_{X_2}$ 
do not increase $\Phi$.
We also have 
$T(V)\subset X_1$ 
and 
$\res{T}{W}=\id_W$.
Since $L_1\ne L_2$, 
$T$ has no fixed points outside~$W$.
Define $L=Z \cap X_1$ and note that $\dim L=1$.
By construction we have $T(v)-v\in Z$ for all $v\in V$,
therefore 
$T(L)\subset L$ 
and moreover 
$T(p+L)\subset p+L$
for every $p\in X_1$.
Pick 
$p \in X_1 \setminus W$ 
and consider the affine map
$\res{T}{p+L}$ 
from the line $p+L$ to itself.
This map cannot be a nontrivial translation of $p+L$ since $T$ does not increase $\Phi$ and 
sublevel sets of $\res{\Phi}{p + L}$ are bounded.
Therefore 
$\res{T}{p+L}$ 
has a fixed point,
hence 
$(p+L)\cap W\ne\varnothing$.
Thus $L\not\subset W$, $X_1=W\oplus L$, and $W\cap Z=0$.

Now we have a projector 
$\pr^Z_W$ 
and it remains to show that it does not increase $\Phi$.
Let $p$ be as above and 
$q=\pr^Z_W(p)$.
Note that $q$ is the unique intersection point of $p+L$ and~$W$ 
hence the unique fixed point of 
$\res{T}{p + L}$.
Since 
$T(L)\subset L$ 
and $T$ does not increase $\Phi$,
the restriction 
$\res{T}{L}$ 
is a multiplication by some 
$\lambda\in[-1,1]$, 
therefore 
$T(p) = T(p - q) + T(q) = \lambda(p - q) + q$.
If $\lambda = 1$ then 
$T(p) = p$ 
which contradicts our choice of $p$ as 
$p \notin W$ 
and $T$ has no fixed points outside of $W$.
If $\lambda=-1$ then 
$q=\frac{p+T(p)}{2}$
and then 
$\Phi(q)\le \Phi(p)$ 
since $\Phi$ is convex and 
$\Phi(T(p))\le\Phi(p)$.
If $|\lambda|<1$ then 
$q=\lim_{m \to \infty} T^m(p)$,
hence 
$\Phi(q)\le \Phi(p)$ 
since $T$ does not increase $\Phi$.

We have shown that 
$\Phi(\pr^Z_W(p))\le\Phi(p)$ 
for an arbitrary 
$p\in X_1\setminus W$.
Thus $\pr^Z_W$ does not increase $\Phi$ on~$X_1$.
Since 
$\pr^Z_W = \pr^Z_W \circ \pr^{L_1}_{X_1}$, 
it follows that $\pr^Z_W$ does not increase~$\Phi$ everywhere.
\end{proof}

\begin{proof}[Proof of Proposition~\ref{p: local Kakutani}]
Recall that Proposition~\ref{p: local Kakutani}
and Theorem~\ref{t: local Blaschke--Kakutani} are equivalent
for every fixed $n$ and~$k$.
The case $n=3$ is covered in section~\ref{s: proof in 3D},
so we assume that $n\ge 4$.
Arguing by induction, we assume that Propositions~\ref{p: local Kakutani}
and Theorem \ref{t: local Blaschke--Kakutani}
are proven for all $3\le n'<n$ in place of~$n$.

Let $V$, $\Phi$, $B$, $X_0$ be as in Proposition~\ref{p: local Kakutani}.
If $B\cap X_0$ is an ellipsoid then the second alternative
of Proposition~\ref{p: local Kakutani} takes place,
so we assume that $B\cap X_0$ is not an ellipsoid.

First assume that $n>k+1$.
For every $W\in\Gr_{k+1}(V,X_0)$
the assumptions of Proposition~\ref{p: local Kakutani}
are satisfied for $W$ in place of~$V$,
$\res{\Phi}{W}$ in place of $\Phi$, and $B\cap W$ in place of~$B$.
Since $B\cap X_0$ is not an ellipsoid,
the $(k+1)$-dimensional case of Proposition~\ref{p: local Kakutani}
implies that $B\cap W$ is locally cylindrical near~$X_0$
for every $W\in\Gr_{k+1}(V,X_0)$.
By Lemma \ref{l: locally cylindrical in subspaces}
it follows that $B$ is locally cylindrical near~$X_0$.
This finishes the proof
of Proposition~\ref{p: local Kakutani} for $n>k+1$.

Now assume that $n=k+1$.
Let $\U$ be a neighborhood of $X_0$ in $\Gr_{n-1}V$
such that all subspaces from $\U$ are $\Phi$-contracting.
We consider two cases.

{\bf Case 1:} All subspaces from $\Gr_{n-2} X_0$ are $\Phi$-contracting.
Then the $(n-1)$-dimensional case of
Theorem~\ref{t: local Blaschke--Kakutani} applies
to $X_0$, $B \cap X_0$, and $\Gr_{n-2} X_0$
in place of $V$, $B$, and $\U$, resp.,
and we conclude that $B \cap X_0$ is an ellipsoid.
(Other possibilities for $B'$ in Theorem~\ref{t: local Blaschke--Kakutani}
are excluded as explained in Remark~\ref{rem: nondegeneracy}).

{\bf Case 2:} There exists $W_0\in\Gr_{n-2} X_0$
that is not $\Phi$-contracting.
Define
$$
 \Sigma = \{ W\in \Gr_{n-2} V : \text{$W$ is not $\Phi$-contracting} \}
$$
Since being $\Phi$-contracting is a closed condition,
$\Sigma$ is an open subset of $\Gr_{n-2} V$.

Let $L_0\in\Gr_1V$ be a contracting direction for $X_0$
(see Definition \ref{d: contracting subspace}).
Pick $X_1\in\U$
such that $X_1\ne X_0$ and 
$X_1\cap X_0=W_0$.
Applying  Lemma \ref{l: dimension reduction} to
the hyperplanes $X_0$ and $X_1$ we conclude that
$L_0$ is the unique contracting direction for $X_1$,
otherwise $W_0$ would be $\Phi$-contracting.

Now consider the set
$$
 \U_0 = \{ X\in\U :\ X\ne X_1 \text{ and } X\cap X_1\in\Sigma \} .
$$
It is an open subset of $\U$ containing $X_0$.
For every $X\in\U_0$ we apply Lemma \ref{l: dimension reduction}
to $X$ and $X_1$ and conclude that contracting directions
for~$X$ and $X_1$ coincide
(since $X\cap X_1$ is not $\Phi$-contracting).
Thus all hyperplanes from $\U_0$ have the same contracting direction $L_0$,
and an application of Lemma~\ref{l: same cylinder}
shows that $B$ is locally cylindrical near~$X_0$.

Thus we have shown that in all cases one of the alternatives
from the conclusion of Proposition~\ref{p: local Kakutani}
holds for an arbitrary $X_0\in\U$.
This finishes the proof of Proposition \ref{p: local Kakutani}
and Theorem \ref{t: local Blaschke--Kakutani}.
\end{proof}

\section{Proof of Theorem \ref{t: local Banach}}
\label{sec: proof_local_Banach}

The proof of Theorem \ref{t: local Banach} is essentially
the same as that of the main result of~\cite{IMN23}
except that the use of the global Kakutani criterion
is replaced by an application of Theorem~\ref{t: local Blaschke--Kakutani}.
Below we go through the steps of the proof from~\cite{IMN23} for $k=3$
and fill out missing bits in the case $k=2$
(which was not considered in \cite{IMN23}).

We restate the key intermediate results from \cite{IMN23}
in the following two propositions.
The first one works in all dimensions and provides a special algebraic family
of tangent directions to~$\pd B$.

\begin{prop}[\cite{IMN23}*{Proposition 2.4} and \cite{IMN23}*{Remark 4.6}]
\label{prop:Rexists}
Let $V$ be a vector space,
$\dim V=k+1\ge 3$, $B\subset V$
a convex body containing 0 in its interior,
and $\U \subset \Gr_k V$ an open set.
Suppose that for every $X_1,X_2\subset\U$
the cross-sections $B\cap X_1$ and $B\cap X_2$
are linearly equivalent.

Then for almost every $X\in\U$ there exist
a vector $\nu\in V\setminus X$
and a linear map
$$
 R\co X^*\to\Hom(X,X)
$$
such that for every $\la\in X^*$
the linear operator $R_\la = R(\la)\co X\to X$ satisfies:
\begin{enumerate}
\item $\tr R_\la=0$.
\item For every $p\in \pd B\cap X$,
the vector $R_\la(p) + \la(p)\nu$ is tangent to $\pd B$ at~$p$.
\end{enumerate}
\end{prop}

The notion of tangency to $\pd B$ in Proposition~\ref{prop:Rexists}(2) is defined as follows:
A vector $v\in V$ is said to be tangent to~$\pd B$ at a point $p\in\pd B$
if for the Minkowski norm $\Phi$ associated to $B$
the function $t\mapsto \Phi(x+tv)$ has zero derivative at $t=0$.
One can see that this is equivalent to the property that
the tangent cone of $B$ at $p$ contains $\linspan\{v\}$.

The most important case in Proposition~\ref{prop:Rexists}(2) is
when $p\in\ker\la$. In this case the term $\la(p)\nu$ vanishes
and hence $R_\la(p)$ is tangent to~$\pd K$ at~$p$
where $K=B\cap X$ is the respective cross-section.
This property is a strong restriction on the pair $(K,R)$
and at least in dimensions $k=2,3$ we have the following.

\begin{prop}[cf.~\cite{IMN23}*{Proposition 2.5}]\label{prop:Rtangent}
Let $X$ be a vector space, $\dim X=k\in\{2,3\}$,
and let $K\subset X$ be a convex body with $0$ in its interior.
Let $R\co X^*\to\Hom(X,X)$ be a linear map such that
for every $\la\in X^*$ the map $R_\la=R(\la)$ satisfies
$\tr R_\la=0$ and
\begin{equation} \label{e:tangency assumption}
    \text{for every } p\in \pd K \cap \ker\la, \text{ the vector } R_\la(p) \text{ is tangent to }\pd K \text{ at~$p$}.
\end{equation}
Then $R_\la(p)$ is tangent to $\pd K$ at $p$
for \textbf{all} $p\in \pd K$ and $\la\in X^*$.
\end{prop}

\begin{proof}
The case $k=3$ is covered by \cite{IMN23}*{Proposition 2.5}.
The proof for $k=2$ can be assembled
from arguments in \cite{IMN23} as follows.

Fix a basis $(e_1,e_2)$ of $X$,
and let $(e_1^*,e_2^*)$ be the dual basis of~$X^*$.
For a point $p=xe_1+ye_2\in X$ define $\la_p=ye_1^*-xe_2^*$
and observe that $p\in\ker\la_p$.
Hence by \eqref{e:tangency assumption}
the vector $W(p):=R_{\la_p}(p)$ is tangent to $\pd K$ at~$p$.
Denote $R_{ij} = R_{e_i^*}(e_j)$
and rewrite $W(p)$ using the linearity of $R$:
\begin{equation}\label{e: W in coordinates}
 W(p) = R_{\la_p}(p) = R_{ye_1^*-xe_2^*}(xe_1+ye_2)
 = xy(R_{11}-R_{22}) -x^2 R_{21} + y^2 R_{12} .
\end{equation}
Thus $W$ is a quadratic vector field on $X$ and it is tangent to $\pd K$
everywhere.
This implies (see \cite{IMN23}*{Lemma 3.4})
that $K$ is a 0-centered ellipse or $W=0$.
In the case of an ellipse the result follows from \cite{IMN23}*{Lemma 2.6},
which is independent of the dimension.

It remains to consider the case $W=0$
(cf.\ \cite{IMN23}*{Lemma 5.2}).
In this case \eqref{e: W in coordinates}
vanishes as a function of $x$ and $y$,
therefore $R_{12}=R_{21}=0$ and $R_{11}=R_{22}$.
Since $\tr R_{e_1^*}=\tr R_{e_2^*} = 0$,
we have $R_{11}^1=-R_{12}^2=0$ and $R_{22}^2=-R_{21}^1=0$
where $R_{ij}^m$, $m=1,2$, denotes the $m$th coordinate of $R_{ij}$
with respect to the basis $(e_1,e_2)$.
Now the identity $R_{11}=R_{12}$ implies that
$R_{11}^2=R_{22}^2=0$ and $R_{22}^1=R_{11}^1=0$.
Thus the tensor $R$ is zero,
hence $R_\la(p)=0$ for all $\la\in X^*$ and $p\in X$,
and the assertion of the proposition follows.
\end{proof}

Now we deduce Theorem \ref{t: local Banach} from
Propositions \ref{prop:Rexists} and \ref{prop:Rtangent};
the argument essentially repeats the one from \cite{IMN23}*{\S 2.3}.

Let $n$, $k$, $V$, $B$ and $\U$ be as in Theorem \ref{t: local Banach}.
First we assume that $n=k+1$
and apply Proposition~\ref{prop:Rexists}.
Let $X$, $R$ and $\nu$ be as in the assertion
of Proposition~\ref{prop:Rexists} and $K=B\cap X$.
Then $K$ and $R$ satisfy the assumptions of
Proposition \ref{prop:Rtangent} and we conclude that
for all $p\in\pd K$ and $\la\in X^*$,
the vector $R_\la(p)$ is tangent to $\pd K$
and hence to $\pd B$ at~$p$.
Pick $p\in\pd K$ and choose $\la\in X^*$
such that $\la(p)\ne 0$.
Now we have two vectors,
$R_\la(p) + \la(p)\nu$ from Proposition~\ref{prop:Rexists}
and $R_\la(p)$ from Proposition~\ref{prop:Rtangent},
such that they are both tangent to $\pd B$ at~$p$
and $\nu$ is their linear combination.
These properties imply that $\nu$ is also tangent to $\pd B$ at~$p$
(a detailed proof of this implication can be found in
\cite{IMN23}*{Lemma 2.3}).

Let $Y=\linspan\{\nu\}$.
The above tangency and convexity of $B$ imply that
$(p + Y) \cap \Int B = \varnothing$
for all $p \in \pd K$.
By Lemma~\ref{l: characterisations of cylindricity}
it follows that $B$ is contained in the cylinder $(B \cap X) + Y$.
Thus we have shown that almost every $X\in\U$ satisfies the
assumption of Theorem~\ref{t: local Blaschke--Kakutani}
(the ``almost every'' is inherited from
Proposition \ref{prop:Rexists}).
This assumption is a closed condition, therefore
it is satisfied for all $X\in\U$.
Now we apply Theorem~\ref{t: local Blaschke--Kakutani}
and conclude that Theorem~\ref{t: local Banach}
holds for $n=k+1$.

It remains to handle the case $n\ge k+2$.
By Lemma \ref{l: relaxed theorem conclusion} it
suffices to verify that for every $X\in\U$,
$B$ is locally cylindrical near $X$ or $B\cap X$ is
an ellipsoid.
Pick $X\in\U$.
For every $W\in\Gr_{k+1}(V,X)$,
the assumption of Theorem \ref{t: local Banach} is
satisfied for $W$ in place of $V$,
$B\cap W$ in place of $B$,
and the connected component of $\U\cap\Gr_k(W)$ containing $X$
in place of~$\U$.
If $B\cap X$ is not an ellipsoid then by
the codimension~1 case of Theorem~\ref{t: local Banach} proven above,
$B\cap W$ is locally cylindrical near $X$ for every $W\in\Gr_{k+1}(V,X)$.
By Lemma \ref{l: locally cylindrical in subspaces} this implies
that $B$ is locally cylindrical near~$X$.
This finishes the proof of Theorem~\ref{t: local Banach}.

\section*{Acknowledgements}
The authors are grateful to the anonymous referee for prompting us to write Remark~\ref{r: local reconstruction from sections}.

\section*{Funding}
This work was supported by the Russian Science Foundation under Grant 21-11-00040. The second author was supported by the Engineering and Physical Sciences Research Council [EP/S021590/1], The EPSRC Centre for Doctoral Training in Geometry and Number Theory (The London School of Geometry and Number Theory), University College London.


\begin{bibdiv}
\begin{biblist}


\bib{AMU35}{article}{
   author={Auerbach, H.},
   author={Mazur, S.},
   author={Ulam, S.},
   title={Sur une propri\'{e}t\'{e} caract\'{e}ristique de l'ellipso\"{\i}de},
   language={French},
   journal={Monatsh. Math. Phys.},
   volume={42},
   date={1935},
   number={1},
   pages={45--48},
   issn={1812-8076},
   review={\MR{1550413}},
}

\bib{Banach32}{book}{
   author = {Banach, Stefan},
   title = {Th{\'e}orie des op{\'e}rations lin{\'e}aires},
   series = {Monografie Matematyczne},
   volume = {1},
   year = {1932},
   publisher = {PWN, Warszawa},
   language = {French},
   reprint ={
      publisher={\'{E}ditions Jacques Gabay, Sceaux},
      date={1993},
      pages={iv+128},
      isbn={2-87647-148-5},
      review={\MR{1357166}},
   },
}

\bib{BHJM21}{article}{
   author={Bor, Gil},
   author={Hern\'{a}ndez Lamoneda, Luis},
   author={Jim\'{e}nez-Desantiago, Valent\'{\i}n},
   author={Montejano, Luis},
   title={On the isometric conjecture of Banach},
   journal={Geom. Topol.},
   volume={25},
   date={2021},
   number={5},
   pages={2621--2642},
   issn={1465-3060},
   review={\MR{4310896}},
}

\bib{Calvert87}{article}{
   author={Calvert, Bruce D.},
   title={Characterizing a space which is Euclidean near a two-dimensional
   subspace},
   journal={Houston J. Math.},
   volume={13},
   date={1987},
   number={3},
   pages={319--335},
   issn={0362-1588},
   review={\MR{0916140}},
}

\bib{LTWZ}{article}{
   author={\v{C}ap, Andreas},
   author={Cowling, Michael G.},
   author={de Mari, Filippo},
   author={Eastwood, Michael},
   author={McCallum, Rupert},
   title={The Heisenberg group, ${\rm SL}(3,\mathbb R)$, and rigidity},
   conference={
      title={Harmonic analysis, group representations, automorphic forms and
      invariant theory},
   },
   book={
      series={Lect. Notes Ser. Inst. Math. Sci. Natl. Univ. Singap.},
      volume={12},
      publisher={World Sci. Publ., Hackensack, NJ},
   },
   isbn={978-981-277-078-3},
   isbn={981-277-078-X},
   date={2007},
   pages={41--52},
   review={\MR{2401809}},
}

\bib{Dvo59}{article}{
   author={Dvoretzky, Aryeh},
   title={A theorem on convex bodies and applications to Banach spaces},
   journal={Proc. Nat. Acad. Sci. U.S.A.},
   volume={45},
   date={1959},
   pages={223--226; erratum, 1554},
   issn={0027-8424},
   review={\MR{105652}},
}

\bib{Gro67}{article}{
   author={Gromov, M. L.},
   title={On a geometric hypothesis of Banach},
   language={Russian},
   journal={Izv. Akad. Nauk SSSR Ser. Mat.},
   volume={31},
   date={1967},
   pages={1105--1114},
   issn={0373-2436},
   review={\MR{0217566}},
   translation={
      journal={Mathematics of the USSR-Izvestiya},
      year={1967},
      volume={1},
      number={5},
      pages={1055--1064}
   },
}

\bib{Gruber}{book}{
   author={Gruber, Peter M.},
   title={Convex and discrete geometry},
   series={Grundlehren der mathematischen Wissenschaften [Fundamental
   Principles of Mathematical Sciences]},
   volume={336},
   publisher={Springer, Berlin},
   date={2007},
   pages={xiv+578},
   isbn={978-3-540-71132-2},
   review={\MR{2335496}},
}

\bib{I18}{article}{
   author={Ivanov, Sergei},
   title={Monochromatic Finsler surfaces and a local ellipsoid
   characterization},
   journal={Proc. Amer. Math. Soc.},
   volume={146},
   date={2018},
   number={4},
   pages={1741--1755},
   issn={0002-9939},
   review={\MR{3754357}},
}

\bib{IMN23}{article}{
   author={Ivanov, Sergei},
   author={Mamaev, Daniil},
   author={Nordskova, Anya},
   title={Banach's isometric subspace problem in dimension four},
   journal={Invent. Math.},
   volume={233},
   date={2023},
   number={3},
   pages={1393--1425},
   issn={0020-9910},
   review={\MR{4623545}},
}

\bib{Kakutani39}{article}{
   author={Kakutani, S.},
   title={Some characterizations of Euclidean space},
   journal={Jpn. J. Math.},
   volume={16},
   date={1939},
   pages={93--97},
   issn={0075-3432},
   review={\MR{895}},
}

\bib{N32}{article}{
  title={Eine Kennzeichnung homothetischer Eifl\"{a}chen},
  author={Nakajima, Soji},
  journal={Tohoku Mathematical Journal, First Series},
  volume={35},
  number={ },
  pages={285--286},
  date={1932},
}

\bib{PS21}{article}{
   author={Purnaras, Ioannis},
   author={Saroglou, Christos},
   title={Functions with isotropic sections},
   journal={Trans. Amer. Math. Soc.},
   volume={374},
   date={2021},
   number={4},
   pages={3007--3024},
   issn={0002-9947},
   review={\MR{4223041}},
}

\bib{S32}{article}{
  title={Zusammensetzung von Eik\"{o}rpen und homothetische Eifl\"{a}chen},
  author={S\"{u}ss, Wilhelm},
  journal={Tohoku Mathematical Journal, First Series},
  volume={35},
  number={ },
  pages={47--50},
  date={1932},
}

\bib{Z18}{article}{
   author={Zhang, Ning},
   title={On bodies with congruent sections or projections},
   journal={J. Differential Equations},
   volume={265},
   date={2018},
   number={5},
   pages={2064--2075},
   issn={0022-0396},
   review={\MR{3800112}},
}

\end{biblist}
\end{bibdiv}
\end{document}